\documentclass{amsart}
\usepackage{amsmath,amsthm,amssymb, cancel}

\theoremstyle{plain}
\newtheorem{theorem}{Theorem}[section]
\newtheorem{lemma}[theorem]{Lemma}
\newtheorem{corollary}[theorem]{Corollary}

\theoremstyle{definition}
\newtheorem{definition}[theorem]{Definition}
\newtheorem{example}[theorem]{Example}

\theoremstyle{remark}

\title{A short nonstandard proof of the Doob-Meyer and Dol{\'e}ans-Dade theorems}

\author{Takashi Matsunaga}
\address{Department of Medical Informatics, Osaka International Cancer Institute, Japan}
\email{matsunaga-ta@nifty.com}

\begin{document}

\renewcommand{\thefootnote}{\fnsymbol{footnote}}

\begin{abstract} 
Using nonstandard analysis, a very short and elementary proof of the Doob-Meyer decomposition and the Dol{\'e}ans Dade theorems is provided. 
\end{abstract}

\maketitle

\section{Introduction}
Although the Doob-Meyer theorem (Theorem \ref{predictable}) is one of the most fundamental theorems in the theory of continuous time stochastic processes, proofs in standard textbooks are long and not straightforward (Bass \cite{Bass} and Kallenberg \cite{Kallenberg}). Recently, Beiglb\"{o}ck \cite{Beigl} proves briefly the exsitence of the Doob-Meyer decomposition for c\`{a}dl\`{a}g submartingales on a probablity space with a right-continuos complete filtration.
 
On the other hand, Lindstr{\o}m \cite{Lindstrom} offers a Doob-Meyer decomposition regarding hyperfinite submartingales, which might yield a nonstandard proof of the Doob-Meyer decomposition for (standard) submartingales. The procedure, however, would be lengthy because one need to establish lifting and pushing down theorems (Hoover and Perkins \cite{Hoover}) and some representation of a standard filtered probability space by a Loeb adapted space (Fajardo and Keisler \cite{Fajardo}).      

By nonstandard analysis, we provide a very short and elemantary proof of both the existence and uniqueness of the Doob-Meyer decomposition for general (not necessarily c\`{a}dl\`{a}g) submartingales as well as the Dol{\'e}ans-Dade theorem as a simple corollary.

\section{Preliminaries}
We assume some familiarity with nonstandard analysis. A quick introduction to nonstandard analysis is presented in the Appendix. Consult Davis \cite{Davis} for details. 

\begin{lemma}\label{lem:standardize}
Let $(M, \mathcal{A}, \nu)$ be a standard finite measure space, let $\mathcal{B}$ be a standard sub $\sigma$-algebra of $\mathcal{A}$, and let X be a nonnegative $^*\mathcal{A}$-measurable function on $^*M$. Suppose  that for any $0 < \epsilon \in \mathbb{R}$, there exists $k \in \mathbb{R}$ such that
   $$ ^*E(X; X \ge k) \equiv \raisebox{2.0ex}{\rm \scriptsize *} \! \! \! 
   \int_{X \ge k} X \ d\nu < \epsilon. $$
Then, there exists a standard almost surely (a.s.) unique nonnegative integrable function $X^\infty$ on $M$ such that
   $$ \ E({X^\infty}; A) \simeq \,\! ^*E(X; \,\! ^*\!A) \equiv
   \raisebox{2.0ex}{\rm \scriptsize *}  \! \! \! \int_{^*\!A} X \ d\nu 
   \ \ \ \ \ \rm( \simeq \ : \ infinitesimaly \ close \ to), $$
   $$ E(Y{X^\infty}; A) \simeq \,\! ^*E(^*YX; \,\! ^*\!A), \ \ \ \ \
   E(E(X^\infty|\mathcal{B}); A)
   \simeq \,\! ^*E(^*E(X|^*\mathcal{B}); \,\! ^*\!A) $$
for all standard $A \in \mathcal{A}$ and all standard bounded measurable function $Y$ on $M$.
\end{lemma}

\begin{proof}
We firstly note that for $A \in \mathcal{A}$ and $0<k \in \mathbb{R}$,
   $$ 0 \le \,\! ^*E(X; \,\!^*\!A) \le \,\! ^*E(X; X \ge k) + k\,^*\!\nu(^*\!A). $$
Hence 
$ 0 \le \,\! ^*E(X; \,\!^*A)$ is finite by assumption so that we can define the standard function $\mu$ on $\mathcal{A}$ by $\mu(A) = {\rm st}(^*E(X; \,\! ^*\!A))$ (st: standard part)
for $A \in \mathcal{A}$. From the above equalities, it is easily seen that $\mu$ is a measure on $\mathcal{A}$, which is absolutely continuous with respect to $\nu$. The Radon-Nikodym theorem yields
 the first formula. The second formula easily follows from the monotone convergence theorem. Finally using the second formula, we have (denoting  the indicator function of $A$ by $\mathbb{I}_A$)
\begin{align*} 
   E(E(X^\infty|\mathcal{B});A)
   &= E(E(X^\infty|\mathcal{B})\mathbb{I}_A) 
   = E(X^\infty E(\mathbb{I}_A|\mathcal{B})) \\
   &\simeq \,\! ^*E(X \, ^*E(\mathbb{I}_{^*\!A}|^*\! \mathcal{B}))
   = \,\! ^*E(^*E(X|^*\mathcal{B})\mathbb{I}_ {^*\!A})
   = \,\! ^*E(^* E(X|^*\mathcal{B}); \,\! ^*\!A). \\
\end{align*}
\end{proof}

\begin{corollary}\label{cor:standardize}
If $X$ is $^*\mathcal{B}$-measurable, then $X^\infty$ is $\mathcal{B}$-measurable.
\end{corollary}

\begin{proof}
By the third fromula of the lemma, for all standard $A \in \mathcal{A}$
  $$ E(E(X^\infty|\mathcal{B}); A)
  \simeq \,\! ^*E(^*E(X|^*\mathcal{B});\,\! ^*\!A)
  =\,\!  ^*E(X; \, ^*\!A) \simeq E(X^\infty; A), $$
which concludes the proof, since $E(X^\infty|\mathcal{B})$ and $X^\infty$ are standard. 
\end{proof}

\section{A version of the Doob-Meyer theorem}
We fix a standard filtered probability space $(\Omega, \mathcal{F}, \{\mathcal{F}_t \ (0 \le t \le1)\}, P)$ throughout this article.

\begin{theorem}(predictable case)\label{predictable}
A standard submartingale $X_t \  (0 \le t \le 1)$ of class D has a decomposition :
   $$ X_t = A_t + M_t \ \ \ (0 \le t \le 1), $$ 
where $A_t$ is a standard almost surely (a.s.) nondecreasing predictable process starting at 0 and $M_t$ is a standard martingale. $A_t$ is a.s. unique.
\end{theorem}

\begin{proof}
For a finite set $ n  \equiv \{ 0=t_0, \ \dots \ , t_k < t_{k+1}, \ \dots \ , t_m = 1 \}$,
define $A^n_t (0 \le t \le1)$ by 
   $$  A^n_t \equiv \sum_{t_{k+1} \le t} E(X_{t_{k+1}}-X_{t_k}|\mathcal{F}_{t_k}). $$
We will show that $A^n_1$'s are uniformly integrable with respect to $n$. Note that for $k >0$
   $$ P(A^n_1 \ge k) \le \frac{1}{k}E(A^n_1)
   = \frac{1}{k}E(X_1-X_0) \ \ \ {\rm (not \ depending \ on \ } n).$$
Let ${\tau}_k = \max \{ t \in n ; A^n_t \le k \}$ for $k >0$ then
\begin{align*}
   E(A^n_1; A^n_1 \ge 2k)
   &\le 2E(A^n_1- \min(A^n_1, k); A^n_1 \ge 2k)
   \le 2E(A^n_1- \min(A^n_1, k)) \\
   &\le 2E(A^n_1- A^n_{{\tau}_k}) = 2E(X_1- X_{{\tau}_k})
   = 2E(X_1- X_{{\tau}_k}; A^n_1 \ge k)  \\
   & \le \epsilon(k) \ \ \ (\epsilon(k) \downarrow 0 \ \ (k \rightarrow \infty),  \ \     \epsilon(k)  {\rm : depending \ only \ on \ } k),
\end{align*}
since $X_t$ is of Class D by assumption.

For a *-finite set $N \equiv \{ 0=t_0, \ \dots \ , t_k < t_{k+1}, \  \dots \ , t_M = 1 \}$, by the Transfer Principle, we can define the *-stochastic process $A^N_t (0 \le t \le 1)$ by 
$$  A^N_t \, \equiv \, ^* \! \! \! \sum_{t_{k+1} \le t} {^*E(^*X_{t_{k+1}} - \,\! ^*X_{t_k}|{^*\mathcal{F}}_{t_k}}) \ \ \ \ \ \ \ (^*\sum \rm{: *-finite \ sum}). $$
For all $0<k \in \mathbb{R}$, $^*E(A^N_1; A^N_1 \ge 2k) \le \,\! ^*\epsilon(k)$ again by the Transfer Principle. By the Concurrence Theorem, we can assume $N$ contains all the standard reals in [0,1]. 

By Lemma \ref{lem:standardize} with $\mathcal{A} = \mathcal{F}$, there exists a standard nonnegative integrable random variable $A_1^\infty $ such that $ E(A_1^\infty; F) \simeq \,\! ^*E(A^N_1; \,\! ^*F) $
for all standard $F \in \mathcal{F}$. Define $M_t$ and $A_t$ by
   $$ M_t \equiv E(X_1-A_1^\infty|\mathcal{F}_t), \ \ \ \ \ A_t
   \equiv X_t - M_t = X_t - E(X_1-A_1^\infty|\mathcal{F}_t). $$ 
By the definition of $A^N_t$, for all $t \in N$ 
   $$ A^N_t = \,\! ^*X_t - \,\! ^*E(^*X_1-A^N_1|^*\mathcal{F}_t). $$
Again using Lemma \ref{lem:standardize} with $\mathcal{A} = \mathcal{F}$  and $\mathcal{B} = \mathcal{F}_t$, we obtain for all standard $t \in N$ and all standard $F \in \mathcal{F}, E(A_t; F) \simeq \,\! ^*E(A^N_t; \,\! ^*F)$.  Since $A^N_t$ is *-a.s. nondecreasing starting at 0 by construction, $A_t$ is also a.s. nondecreasing starting at 0.
Hence for all standard $0 \le p < q \le 1$ and all standard $F \in \mathcal{F}$,
   $$ \int_{[p,q) \times F} A_t \ dtdP \simeq 
   \raisebox{2.0ex}{\rm \scriptsize *} \! \! \! 
   \int_{^*[p, q) \times ^*\!F} A^N_t \ dtdP, $$
by the Fubini theorem and Transfer Principle. The monotone class theorem yields that for all $A \in \mathcal{B}[0,1] \otimes \mathcal{F}, $ 
   $$ \int_A A_t \ dtdP \simeq 
   \raisebox{2.0ex}{\rm \scriptsize *} \! \! \! 
   \int_{^*\!A} A^N_t \ dtdP, $$
where $\mathcal{B}[0,1]$ is the Borel $\sigma$-field on [0,1]. 
Since $A^N_t$ is *-predictable by construction, we see that $A_t$ is also predictable by applying Corollary \ref{cor:standardize} with $\mathcal{A} = \mathcal{B}[0,1] \otimes \mathcal{F}$ and $\mathcal{B} = \mathcal{P}$, where $\mathcal{P}$ is the predictable $\sigma$-field (the $\sigma$-field generated by the set $ \{ \{0\} \times F_0, \  (s,t] \times F_s \ | \ F_0 \in \mathcal{F}_0, \ F_s \in \mathcal{F}_s, \  0 \le s < t \le1 \} $).    

Finally,  we will show the uniqueness of the decomposition. Let another decomposition be
$ X_t = A^{'}_t + M^{'}_t. $ Here $M^{'}_t$ is a martingale, so that 
   $$ A^N_t \equiv \,\! 
   ^*\!\!\! \sum_{t_{k+1} \le t} \,\! ^*E(^*X_{t_{k+1}}- \,\! ^*X_{t_k}|
   ^*\mathcal{F}_{t_k})  
   = \,\! ^*\!\!\! \sum_{t_{k+1} \le t} \,\! ^*E(^*A^{'}_{t_{k+1}}
    \, -\,  ^*A^{'}_{t_k}|
   ^*\mathcal{F}_{t_k}). $$
Since $N$ contains all the standard reals in [0,1], for all standard $r, t\in [0,1]$ with $r < t $ and all standard $F \in \mathcal{F}_r $,
   $$ E(A_t; F) \!\simeq\! ^*\!E(A^N_t; \,\! ^*\!F)
   \!=\! ^*\!E( ^*\!\!\! \sum_{t_{k+1} \le t}
   \!\!\!^*\!E(^*\!A^{'}_{t_{k+1}} -\! ^*\!A^{'}_{t_k}|
   ^*\!\mathcal{F}_r); \, ^*\!F) \!=\! ^*\!E(^*\!A^{'}_t; \,\! ^*\!F)
   \!=\! E(A^{'}_t; F), $$
which shows $A_t = A^{'}_t$ a.s. because both $A_t$ and $A^{'}_t$ are standard and predictable so that they are $\mathcal{F}_{t-}$measurable.
\end{proof}

\begin{corollary}(natural case)
$A_t$ is also a natural process and a.s. unique.
\end{corollary}

\begin{proof}
For all standard bounded martingale $N_t$,
\begin{align*}
   E(\int_0^1 N_{t-}dA_t)
   &\simeq \!^*\!\sum_k \!^*\!E(\!^*\!N_{t_k}(\!^*\!A_{t_{k+1}} - \!^*\!A_{t_k})) 
   = \!^*\!\sum_k \!^*\!E(\!^*\!N_{t_k} \!^*\!E(\!^*\!A_{t_{k+1}}- \!^*\!A_{t_k}|
  ^*\!\mathcal{F}_{t_k})) \\
   &= \!^*\!\sum_k \!^*\!E(\!^*\!N_{t_k} \!^*\!E(\!^*\!X_{t_{k+1}}
   - \!^*\!X_{t_k}|\!^*\!\mathcal{F}_{t_k}))
   = \!^*\!\sum_k \!^*\!E(\!^*\!N_{t_k}(A^N_{t_{k+1}}-A^N_{t_k})) \\
   &= \!^*\sum_k \!^*\!E^* (\!^*\!N_{t_{k+1}}A^N_{t_{k+1}} - \!^*\!N_{t_k}
   A^N_{t_k})
   \ \ \ \ \ \ \ {N_t \rm : \ martingale} \\
   &= \!^*\!E(\!^*\!N_1A^N_1)  \simeq E(N_1A_1),
   \ \ \ \ \ \ \ {\rm Lemma \ \ref{lem:standardize}.} 
\end{align*}
Thus we have $E(\int_0^1 N_{t-}dA_t)=E(N_1A_1)$ as required. For the uniqueness, let another standard natural process be $A^{'}_t$. Recall that $N$ contains all the standard reals in [0,1]. Note that from the above expressions, for all standard bounded martingale $N_t$, $ E(N_1A_1) = E(N_1A^{'}_1). $ Hence $A_1 = A^{'}_1$  a.s., which in turn means $A_t = A^{'}_t$ a.s. because
   $$ A^{'}_t = X_t - E(X_1-A^{'}_1|\mathcal{F}_t). $$ 
\end{proof}

\begin{corollary}(Dol{\'e}ans-Dade)
A standard integrable nondecreasing process $A_t \ (0 \le t \le 1)$ is predictable if and only if it is natural.
\end{corollary}  

\begin{proof}
Note that $A_t$ is a submartingale of Class D. Suppose that $A_t$ is predictable or natural. The uniqueness of the decomposition tells us that for all standard $F \in \mathcal{F}$ and all standard $t \in [0,1]$, 
$$ E(A_t; F) \simeq \,\!^*\!E(A^N_t; \,\!^*F), $$
where $A^N_t$ is constructed in the proof of the theorem. The rest is easy.
\end{proof}

\section{Concluding Remarks}
Although we follow Rao's \cite{Rao} idea, we do not need the Dunford-Pettis theorem or Komlos's lemma thanks to nonstandard analysis. The  proofs of the predictability, the uniqueness, and the Dol{\'e}ans-Dade theorem are new and straightforward again owing to nonstandard analysis.  

\subsection*{Disclosure Statement and Funding}
There are no interests to declare. No funding was received.

\section{Appendix: a quick introduction to nonstandard analysis }
Nonstandard analysis is a theory founded by Abraham Robinson in the 1960s, motivated largely by the revival of Leibnizian infinitesimals. He constructed a proper extension of $\mathbb{R}$ denoted by $^*\mathbb{R}$, which is logically similar to $\mathbb{R}$ but includes ideal elements such as infinitesimals and infinite numbers. In essence, he  constructed a nonstandard extension $^*U (\ni {^*\mathbb{R}})$ (Theorem \ref{extension}) of a universe (Definition \ref{universe}) $U (\ni \mathbb{R})$, where $^*U$ is logically similar to $U$ but includes ideal elements. 

We need several definitions and lemmas to prove Thoerem \ref{extension}.

\begin{definition}
A (logical) formula is defined  recursively as follows: 
\begin{enumerate}
\item For variables or constants $u$ and $v$, $u=v$ and $u \in v$ are (atomic) formulae,
\item If $\phi$ is a formula, $\lnot\phi$, $\exists x \phi$, and $\forall x \phi$ are formulae, 
\item If $\phi$ and $\psi$ are formulae, $\phi \land \psi$, $\phi \lor \psi$, and $\phi \to \psi \equiv \lnot\phi \lor \psi $ are formlae, 
\item Only those obtained by the above rules are formlae.
\end{enumerate}
A bounded variable $x$ in $\phi$ is a variable that appears in the form of  $\exists x \phi$ or $\forall x \phi$. Other variables in $\phi$ are free variables. A sentence is a formula without free variables.
\end{definition}

\begin{definition}\label{universe}
A universe $U$ is a set satisfying the following conditions:
\begin{enumerate}
\item $u \in v$ and $v \in U$ imply $u \in U$, 
\item $u \in U$ and $v \in U$ imply $\{u, v\} \in U$,
\item $u \in U$ implies $\bigcup u \in U$,
\item $u \in U$ implies $\mathcal{P}(u) \in U$, where $\mathcal{P}(u)$ is the power set of $u$. 
\end{enumerate}
\end{definition}

\begin{example}(Universe containing $\mathbb{R}$)
Let $V_0 = \mathbb{R}$ and define $V_{n+1}$ by $V_{n+1} = \bigcup V_n$ inductively. Set $U_0 = \bigcup V_n$. Define $U_{n+1}$ by $U_{n+1} = U_n \cup \mathcal{P}(U_n)$ inductively. Set $U = \bigcup U_n$. It is straightforward to verify that $U$ is a universe.
\end{example}
Note that $U$ contains various $\mathbb{R}$-related objects such as real numbers themselves, functions on $\mathbb{R}$, and binary relations on $\mathbb{R}$ and so on. 
Usually, we adopt a universe $U$ that has all relevant mathematical objects.

\begin{definition}
A filter basis $\mathcal{F}$ on a set $I$ is a subset of $\mathcal{P}(I)$ satisfiying (1) and (2). A filter $\mathcal{F}$ is a subset satisfying (1), (2) and (3). An ultrafilter $\mathcal{F}$ is a subset satisfying (1), (2), (3) and (4). 
\begin{enumerate}
\item $\phi \not\in \mathcal{F}$ and $I \in \mathcal{F}$, 
\item If $A \in \mathcal{F}$ and $B \in \mathcal{F}$, then $A \cap B \in \mathcal{F}$,
\item If  $A \in \mathcal{F}$ and $A \subseteq B$, then $B \in \mathcal{F}$,
\item If $\mathcal{G}$ is a filter and $\mathcal{F} \subseteq \mathcal{G}$, then $\mathcal{F} = \mathcal{G}$, that is, $\mathcal{F}$ is maximal under $\subseteq$.
\end{enumerate}
\end{definition}

\begin{lemma}\label{ultraexist}
For a filter basis $\mathcal{F}_0$ on $I$, there exists an ultrafilter $\mathcal{F}$ containing $\mathcal{F}_0$.
\end{lemma}
\begin{proof}
Let $\mathcal{F}_1 = \{ X \in \mathcal{P}(I) | X \supseteq A \ {\rm for \ some} \ A \in \mathcal{F}_0 \}$. It is easy to verify that $\mathcal{F}_1$ is a filter. If $\mathcal{F}_1$ is not maximal, there exsits a filter $\mathcal{F}_2 \supsetneq \mathcal{F}_1$. If $\mathcal{F}_2$ is not maximal, there exsits a filter $\mathcal{F}_3 \supsetneq \mathcal{F}_2$ and so on. Finally we have a maximal $\mathcal{F}$. More formally, the existence of $\mathcal{F}$ follows by Zorn's lemma. 
\end{proof}

\begin{lemma}\label{either}
For an ultrafilter $\mathcal{F}$ on $I$ and $A \subseteq I$, either $A \in \mathcal{F}$ or   
$I - A \in \mathcal{F}$ holds.
\end{lemma}
\begin{proof}
Suppose that $A \not\in \mathcal{F}$. Let $\mathcal{G}  = \{X \in \mathcal{P}(I) | X \cup A \in \mathcal{F} \}$. It is a routine to check that $\mathcal{G}$ is a filter containing $\mathcal{F}$. Hence $\mathcal{G} = \mathcal{F}$ by hypothesis. Obviously $I-A \in \mathcal{G}$. This completes the proof. 
\end{proof}

\begin{definition}(Ultrapower)
Let $V$ be an infinite set and $I$ be an infinite index set. Let denote the set of all maps from $I$ to $V$ by $V^I$. Let $<a(i)>, <b(i)> \in V^I$. Define the equivalence relation on $V^I$ by $\{i \in I | a(i) =b(i) \} \in \mathcal{F}$ (this is well-defined if $\mathcal{F}$ is an ultrafilter on $I$). The ultrapower of $V$ over an ultrafilter $\mathcal{F}$ on the index set $I$ is the set of equivalence classes of $V^I$.
\end{definition}

\begin{theorem}\label{extension}(Nonstandard Extension)
For a given universe $U$, there exsits a nonstandard extension $^*U$ and a binary relation $^* \! \! \in$ on it that satisfy the following two conditions:
\begin{enumerate}
\item the Transfer Principle: There exists an injective map $^*: U \rightarrow {^*U}$ such that any sentence $\phi$ in $U$ holds if only if the "corresponding" sentence $^*\phi$ holds in $^*U$. Here the "corresponding" sentence $^*\phi$ is defined by replacing $\in$ with $^*\!\in$ and $c$'s (constants) with $^*c$'s in $\phi$,
\item the Concurrence Principle: For a formula $\phi(a,b)$ "concurrent" with respect to $A (\in U)$, then there exists $b \in {^*U}$ such that $^*\phi(\,^*a, b)$ holds for all $a \in A$. Here "conccurent" with respect to $A$ means that for any finite collection $a_i \in A (\in U)$, there exists $b \in U$ such that $\phi(a_i, b)$ holds for all $i$.
\end{enumerate}
\end{theorem}
\begin{proof}
The construction of $^*U$ proceeds as follows. Consider the index set $I$ of all finite subsets of $U$. For each $i \in I$, let $\mu(i) = \{ j \in I \mid i \subseteq j \}$. Since $\mu(i_1) \cap \mu(i_2) \cap \cdots \cap \mu(i_n) = \mu(i_1 \cup i_2 \cup \cdots \cup i_n)$, this family is a filter basis on $I$ so that  there exists an ultrafilter $\mathcal{F}$ on $I$ containing all the $\mu(i)$'s by Lemma \ref{ultraexist}.  Denote the ultrapower of $U$ over $\mathcal{F}$ by $^*U$ and the equivalence class by $[a(i)], [b(i)]$. The binary relation $[a(i)] \, {^*\!\!\in}\ [b(i)]$ is also well-defined by $\{ i \in I \mid a(i) \in b(i) \} \in \mathcal{F}$. The map * is defined by $a \mapsto [a(i) = a]$. This map is injective by definition.

The Transefer Principle is the special case of the next lemma. For the Concurrence Principle, by assumption, there exists $b(i)$ such that $\phi(a, b(i))$ holds for all $a \in i \cap A$, we obtain the desired result again using the next lemma.
\end{proof}

\begin{lemma}(\cancel{L}os's Theorem)
Under the same conditions as in Theorem \ref{extension}, for any formula $\phi(x,y,\dots,z)$, $^*\phi([a(i)], [b(i)], \cdots, [c(i)])$ holds if only if
 
$\{i \in I \ |\ \phi(a(i), b(i), \cdots ,c(i)) {\rm \ holds} \} \in \mathcal{F}$, where $x, y, \cdots, z$ are free variables.
\end{lemma}
\begin{proof}
We apply mathematical induction on the number of logical symbols in $\phi(x,y, \dots, z)$. First, if there is no logical symbols in $\phi(x,y,\cdots, z)$, $\phi(x,y,\cdots, z)$ is $x\in y$ or $x=y$ so that the conclusion follows by definiton. Next, using De-Morgan's law, it suffice to prove the cases  of $\exists, \land$ and $\lnot$. Consider the case of $\exists$. Clearly, $\exists z \phi([a(i)],[b(i)], \cdots, z)$ holds if only if for some ${c(i)}$   $\phi([a(i)],[b(i)], \cdots, [c(i)])$ holds. By the hypothesis of induction, this occurs if only if $\{i \in I \ |\ \phi(a(i), b(i),$

$\cdots, c(i)) {\rm \ holds} \} \in \mathcal{F}$. The other cases are similar and left to the reader. The key to the proofs is Lemma \ref{either}.  
\end{proof}

\begin{definition}(Standard, Internal, External)
An entity $u \in U$ and the correponding $^*u \in {^*U}$ are called standard. An entity $v \in {^*U}$ such that  $v \,^*\!\in {^*u}$ for some standard $u$ is called internal. Otherwise $v$ is called external.
\end{definition} 

\begin{example}(Transfer Principle I: Embedding, Exstension, *-Omission)
\begin{enumerate}
\item If $A \in U$ and $a \in A$, then $a \in U$ by the transitivity of $U$. Thus, $^*a$ is defined and $^*a  {\,^*\!\in} {\,^*\!A}$. Since the map $^*: U \mapsto \,^*U$ is injective, $A$ is embedded into $^*A$. We often assume $A \subseteq {^*A}$, if there is no confusion,
\item For $a, b \in \mathbb{R}$ $a<b$ if only if ${^*a} {\ ^*\!\!<} {\ ^*b}$. If we assume $\mathbb{R} \subseteq {^*\mathbb{R}}$ as in (1), $^*\!\!<$ is regarded as an extension of $<$ so that $^*$ is often omitted,
\item For $A, B \in U$ and $f: A \mapsto B$, $b=f(a)$ if and only if $^*b= {^*f(^*a)}$. If we assume $A \subseteq {^*A}$ and $B \subseteq {^*B}$ as in (1), $^*f$ is again an extension of $f$ so that $^*$ is often omitted.
\item For a function $f$ on $\mathbb{R}$,
$$\forall x \in \mathbb{R} \, \forall y \in \mathbb{R} \ f(x+y) = f(x) + f(y)$$
if only if 
$$\forall x \,^*\!\!\in {^*\mathbb{R}} \, \forall y \,^*\!\!\in {^*\mathbb{R}} \ \ {^*f}(x \,^*\!\!+ y) = {^*f}(x) \,^*\!\!+ {^*f}(y).$$
We often omit * from $^*f$ and $^*+$.

\end{enumerate}
\end{example}

\begin{definition}(Transfer Principle II: *-Property, *-Finitenss, *-Finite Sum)
\begin{enumerate}
\item Let $P (\in U)$ define some property $Prop$. We say $u$ is $Prop$ if $u \in P$. In this situation, we say $v$ is *-$Prop$ if $v \,^*\!\in \,^*P$. An example is the following.  For $A \in U$, if $P \equiv \mathcal{P}_F(A)$ (the set of all the finite subsets of $A$), then $u (\in P)$ is a finite subset of $A$ and $v (\,^*\!\in \,^*P)$ is a *-finite subset of $^*A$.
\item Denote the finite sum of the elements of a finite subset of $\mathbb{R}$ by $\Sigma: \mathcal{P}_F(\mathbb{R}) \mapsto \mathbb{R}$. Then, we obtain the *-finite sum $^*\Sigma: \,^*\mathcal{P}_F(\mathbb{R}) \mapsto \ ^*\mathbb{R}$. That is, *-finite sum is defined on all the *-finite subset of $^*\mathbb{R}$.
\end{enumerate}
\end{definition}
 
\begin{example}(Concurrence Principle)
\begin{enumerate}
\item Since the formula $\phi(a,b) \equiv a<b \land b \in \mathbb{R}$ is concurrent with respect to $\mathbb{R}$, we obtain $b \,^*\!\in \,^*\mathbb{R}$ such that for any $a \in \mathbb{R}$ $^*a {\,^*\!\!<}  b$. That is, $b$ is an infinite number and $1/b$ is an infinitesimal. For $x, y \,^*\!\in \,^*\mathbb{R}$, we write $x \simeq y$ if $x-y$ is infinitesimal.
\item The formula $\phi(a,b) \equiv a \in b \land b \in \mathcal{P}_F([0,1])$ is concurrent with respect to $[0,1]$. Hence, there exists $b \,^*\!\in \,^*\mathcal{P}_F([0,1])$ such that $^*a \,^*\!\in b$ for all $a \in [0,1]$. In other words, there exists
*-finite subset of $^*[0,1]$ that contains all the elements of $[0,1]$, if we assume $[0,1] \subseteq \,^*[0,1]$.
\end{enumerate}
\end{example}

\begin{example}(Standard Part)
If $c {\,^*\!\in} {\,^*\mathbb{R}}$ is finite, $\sup(\{ x \in \mathbb{R} \,|\, ^*x  \,^*\!\!<  c \} )$ is called the standard part of $c$ and denoted by ${\rm st}(b)$. It is easy to check $c \simeq {\rm st}(b)$.
\end{example}

\begin{lemma}(Uniform Continuity)\label{uniformcontinuity} 
Let $f(x)$ be a function on $[0,1]$. Then  $f(x)$ is uniformly contiuous on $[0,1]$ if only if $\forall x,y {\,^*\!\in} {^*[0,1]} \  (x \simeq y \to f(x) \simeq f(y))$.
\end{lemma}
\begin{proof}
Suppose that $f(x)$ is uniformly contiuous on $[0,1]$. Then by definition $\forall \epsilon \in \mathbb{R}_+ \exists \delta \in \mathbb{R}_+ \ (|x-y|<\delta \to |f(x)-f(y)|<\epsilon)$. Fix any $\epsilon \in \mathbb{R}_+$. For $x,y \,^*\!\in {^*[0,1]}$, if $x \simeq y$ then obviously $|x-y| <\delta$. By the Transfer Principle $|f(x)-f(y)|<\epsilon$ so that $f(x)\simeq f(y)$ because $\epsilon \in \mathbb{R}_+$ is arbitary. Conversely, suppose that $\forall x,y {\,^*\!\in} {^*[0,1]}\  (x \simeq y \to f(x) \simeq f(y))$. Fix any $\epsilon \in \mathbb{R}_+$. If $|x-y|$ is less than some positive infinitesimal, $|f(x)-f(y)|<\epsilon$ by assumption. That is $\exists \delta {\,^*\!\in} {\,^*\mathbb{R}_+} (|x-y|<\delta \to |f(x)-f(y)|<\epsilon)$. By the Transfer Principle we obtain $\exists \delta \in \mathbb{R}_+ (|x-y|<\delta \to |f(x)-f(y)|<\epsilon)$.
\end{proof}

\begin{definition}(Good *-Partition of $[0,1]$) 
$p \equiv \{0=a_0<a_1<\cdots<a_i<\cdots<a_n =1\} \ (a_i \in \mathbb{R},n \in \mathbb{N})$ is a partition of $[0,1]$. By the Transfer Principle, $P\equiv \{0=a_0<a_1<\cdots<a_i<\cdots<a_N =1\} \ (a_i \in {^*\mathbb{R}}, N \in {^*\mathbb{N}})$
 is a *-partition of $^*[0,1]$. $P$ is called "good" if $\mathbb{R} \subseteq P$ The term "good" is used only in the next example. 
\end{definition}

\begin{example}(Riemann-Stieltjes Integral)
Let $f(x): [0,1] \mapsto \mathbb{R}$ be a continuous function and $g(x): [0,1] \mapsto \mathbb{R}$ be a non-decreasing function. For $p \equiv \{0=a_0<a_1<\cdots<a_n =1\}$ (a partition of $[0,1]$), set $ S(p) \equiv \sum_{k=1}^n f(a_{k-1})(g(a_k)-g(a_{k-1}))$. $S$ is a function from the set of all the partition of [0,1] $I$ to $\mathbb{R}$ so that by the Transfer Principle, $^*S: {^*I} \mapsto {^*\mathbb{R}}$. In other words, for a *-partition of $^*[0,1], P \equiv \{0=a_0 < a_1< \cdots <a_N=1\}$ $^*S(P) = {^*\sum}_{k=1}^N f(a_{k-1})(g(a_k) -g(a_{k-1}))$, where $N \in {^*\mathbb{N}}$ and $^*\sum$ is *-finite sum. Note that *'s are omitted from $^*f, ^*g$ and $^*-$. Suppose that $P$ and $P'$ are "good" *-partitions of $^*[0,1]$. Then $^*S(P) \simeq {^*S(P')}$. To see this, let $P''$ be the combined *-partitions of $P$ and $P'$. Then, it is a routine to verify that $S(P) \simeq S(P'')$ and $S(P') \simeq S(P'')$ by using Lemma \ref{uniformcontinuity}.  
\end{example}

\end{document}